\documentclass[12pt]{article}

\usepackage{stmaryrd}
\usepackage{mathrsfs}
\usepackage{amsmath,amsthm,amssymb}
\usepackage[colorlinks,
            linkcolor=red,
            anchorcolor=blue,
            citecolor=magenta
            ]{hyperref}
\usepackage{pdfsync}
\usepackage[numbers,sort&compress]{natbib}
\allowdisplaybreaks

\theoremstyle{definition}

\theoremstyle{plain}
\newtheorem{theorem}{Theorem}[section]
\newtheorem{definition}{Definition}[section]
\newtheorem{remark}{Remark}[section]
\newtheorem{lemma}{Lemma}[section]

\newtheorem{proposition}{Proposition}[section]

\numberwithin{equation}{section}

\newcommand{\vs}{\vspace}

\begin{document}

\title{Orbital stability of normalized ground states for  critical Choquard equation with potential\footnote{  This work was partially supported by NNSFC (No. 12171493).}}

\author{ Jun Wang$^{a}$, Zhaoyang Yin$^{a, b}$\footnote {Corresponding author. wangj937@mail2.sysu.edu.cn (J. Wang), mcsyzy@mail.sysu.edu.cn (Z. Yin)
} \\
{\small $^{a}$Department of Mathematics, Sun Yat-sen University, Guangzhou, 510275, China } \\
{\small $^{b}$School of Science, Shenzhen Campus of Sun Yat-sen University, Shenzhen, 518107, China } \\
}

	\date{}

	\maketitle

\date{}

 \maketitle \vs{-.7cm}

  \begin{abstract}
In this paper, we
study the existence of ground state standing waves and orbital stability, of prescribed mass,  for the nonlinear critical Choquard equation
\begin{equation*}
  \left\{\begin{array}{l}
i \partial_t u+\Delta u -V(x)u+(I_{\alpha}\ast|u|^{q})|u|^{q-2}u+(I_{\alpha}\ast|u|^{2_{\alpha}^*})|u|^{2_{\alpha}^*-2}u=0,\ (x, t) \in \mathbb{R}^d \times \mathbb{R}, \\
\left.u\right|_{t=0}=\varphi \in  H ^1(\mathbb{R}^d),
\end{array}\right.
\end{equation*}
where $I_{\alpha}$ is a Riesz potential of order $\alpha\in(0,d),\ d\geq3,\  2_{\alpha}^*=\frac{2d-\alpha}{d-2}$ is the upper critical exponent due to Hardy-Littlewood-Sobolev inequality, $\frac{2d-\alpha}{d}<q<\frac{2d-\alpha+2}{d}$. Under appropriate potential conditions, we obtain new Strichartz estimates and construct the new space to get orbital stability of normalized ground state. To our best knowledge, this is the first orbital stability result for this model. Our method is also applicable to other mixed nonlinear equations with potential.
 \end{abstract}

{\footnotesize {\bf   Keywords:} Choquard equation, Normalized solution, Ground state, Orbital stability

}

\section{ Introduction and main results}
In this paper, we consider the following nonlinear critical Choquard equation
\begin{equation}\label{eq1.1}
  i \partial_t u+\Delta u -V(x)u+(I_{\alpha}\ast|u|^{q})|u|^{q-2}u+(I_{\alpha}\ast|u|^{2_{\alpha}^*})|u|^{2_{\alpha}^*-2}u=0,\ (x, t) \in \mathbb{R}^d \times \mathbb{R},
\end{equation}
where $\alpha\in(0,d),\ \frac{2d-\alpha}{d}<q<\frac{2d-\alpha+2}{d},\ 2_{\alpha}^*=\frac{2d-\alpha}{d-2}$. $I_\alpha$ is the Riesz potential defined by
$$
I_\alpha(x)=\frac{\Gamma\left(\frac{d-\alpha}{2}\right)}{\Gamma\left(\frac{\alpha}{2}\right) \pi^{\frac{d}{2}} 2^\alpha|x|^{d-\alpha}} \triangleq \frac{\bar{C}}{|x|^{d-\alpha}}
$$
and $\Gamma$ is the Gamma function. We recall that standing waves to \eqref{eq1.1} are solutions of the form $v(x,t)=e^{-i \lambda t}u(x),\ \lambda\in\mathbb{R}$. Then the function $u(x)$ satisfies the equation
\begin{equation}\label{eq1.2}
  -\Delta u+V(x)u-\lambda u=(I_{\alpha}\ast|u|^{q})|u|^{q-2}u+(I_{\alpha}\ast|u|^{2_{\alpha}^*})|u|^{2_{\alpha}^*-2}u,\ x \in \mathbb{R}^d .
\end{equation}
One can search for solutions to \eqref{eq1.2} having a prescribed $L^2$-norm. Defining on $H^1(\mathbb{R}^d, \mathbb{R})$ the energy functional
\begin{equation*}
  I(u)=\frac{1}{2}\|\nabla u\|_2^2+\frac{1}{2}\int_{\mathbb{R}^d}V(x)u^2dx -\frac{1}{2 q}\int_{\mathbb{R}^d}(I_{\alpha}\ast|u|^{q})|u|^{q}dx-\frac{1}{2\cdot2_\alpha^*}\int_{\mathbb{R}^d}(I_{\alpha}\ast|u|^{2_{\alpha}^*})|u|^{2_{\alpha}^*}dx,
\end{equation*}
it is standard to check that $I(u)$ is of class $C^1$ and that a critical point of $I$ restricted to the (mass) constraint
\begin{equation*}
  S(a) := \{u\in H^1(\mathbb{R}^d, \mathbb{R}) : \|u\|_2^2 = a\}
\end{equation*}
gives rise to a solution to \eqref{eq1.2}, satisfying $\|u\|_2^2 = a$.

It is widely known that Fr\"ohlich \cite{HF1937} and Pekar \cite{SI1954}  firstly introduced the Choquard equation  for the modeling of quantum polaron
\begin{equation*}
  -\Delta u+u=\left(\frac{1}{|x|}*|u|^2\right)u \text{  in } \mathbb{R}^3.
\end{equation*}
As pointed out by Fr\"ohlich and Pekar, this model corresponds to the study of free
electrons in an ionic lattice interacting with phonons associated to deformations of
the lattice or with the polarization that it creates on the medium.

If potential $V(x)$ in \eqref{eq1.1} is constant, we call \eqref{eq1.1} is autonomous. In this case, Jeanjean in \cite{LJJ1997} developed an approach based on the
Pohozaev identity which has been used successfully in recent years. The key to this method is to find a bounded Palais-Smale sequences by using the transformation $s*u(x) = e^{\frac{sd}{2}}u(e^sx)$. After that, by weakening the conditions in \cite{LJJ1997}, Jeanjean \cite{LJSL2020} and Bieganowski \cite{BBJM2021} improved these results. Of course, these articles only consider the problem of a single nonlinear term. Recently, there have been many studies on mixed nonlinear terms. For example, Soave \cite{{NSJDE2020},{NSJFA2020}} studied normalized solution of \eqref{eq1.2}
with mixed nonlinearity $f(|u|)u = \mu|u|^{q-2}u + |u|^{p-2}u$, $2 < p < 2 + \frac{4}{d} < q \leq2^*=\frac{2d}{d-2}$. Specifically, Soave in \cite{NSJDE2020} obtained many results of existence and non-existence. More precisely, if $2<q<p=2+\frac{4}{d}$, that is, the
leading nonlinearity is $L^2$-critical and a $L^2$-subcritical lower order term. \eqref{eq1.2} had a real-valued positive and radially symmetric solution for some $\lambda<0$ in $\mathbb{R}^d$ provided $\mu>0$ and $\Theta>0$ small enough. Moreover, if $\mu<0$, \eqref{eq1.2} had no solution. If $2+\frac{4}{d}=q<p<2^*$, that is, the
leading term is $L^2$-critical and $L^2$-supercritical, \eqref{eq1.2} had a real-valued positive, radially symmetric solution for some $\lambda<0$ in $\mathbb{R}^d$ provided $\mu>0$ and $\mu, \Theta$ satisfy the appropriate conditions. If $2<q< 2+\frac{4}{d}<p<2^*$, that is, the
leading term $L^2$-subcritical and $L^2$-supercritical, \eqref{eq1.2} also had a real-valued positive and radially symmetric solution for some $\lambda<0$ in $\mathbb{R}^d$ provided $\Theta>0, \mu<0$ and $\mu, \Theta$ satisfy the appropriate conditions. Soave in \cite{NSJFA2020} considered the Sobolev critical case and obtained some similar results. In particular, the Sobolev critical case also has been considered in \cite{{TAK2012},{TAK2013},{RKTO2017},{CMGX2013}}(see also the references therein). It is worth mentioning that many researchers are also interested in the existence of normalized multiple solutions. In \cite{LJTL2022}, Jeanjean et al. obtained the existence of normalized multiple solutions for Sobolev critical case in \eqref{eq1.2}. For more results on this aspect, please refer to \cite{{JWYW2022},{TBNS2017},{TBNS2019},{DBJC2019},{JBLT2013},{TBAQ2023}} and its references.

If \eqref{eq1.2} is non-autonomous, there is little research on normalized solutions for Choquard equation with potential. In this paper, we are interested in existence of ground state solutions and orbital stability. Regarding stability research, we mention the guiding work of Jeanjean \cite{LJJJ2022} here. Jeanjean et al. in \cite{LJJJ2022} first derive an upper bound on the propagator $e^{it\Delta}$ which provides a kind of uniform local existence result, using the information that all minimizing sequences to show  the orbital stability of ground state solutions. Our idea is to generalize its results to critical Choquard equation with potential. However, there are two fundamental difficulties. On the one hand, due to the emergence of the potential function, the original operator $\Delta$ becomes $\Delta-V$, so the corresponding Strichartz may not  hold true. On the other hand, we need to construct a new space to obtain local solutions since the emergence of the $V(x)$. To overcome the first difficulty, we need to obtain new Strichartz estimates by establishing new dispersion estimates and dual methods in the case of Kato potential.
To overcome the second difficulty, we use a new equivalent norm to construct space and distance, see Lemma \ref{L3.3}. Furthermore, dealing with Choquard nonlinear terms is not easy, and we need to establish new estimates based on fundamental inequality, such as Proposition \ref{P3.3}.

Now, we give the definition of ground state solutions and orbital stability.
\begin{definition}\label{D1.1}
We say that $u_a\in S(a)$ is a ground state solution to \eqref{eq1.2} if it is a solution having minimal
energy among all the solutions which belong to $S(a)$. Namely, if
\begin{equation*}
  I(u_a)=\inf\{I(u),\ u\in S(a),\ (I\left.\right|_{S(a)}'(u))=0\}.
\end{equation*}
\end{definition}
\begin{definition}\label{D1.2}
$Z\subset \mathcal{H}$ is stable if: $Z\neq\emptyset$ and for any $v\in Z$ and any $\varepsilon>0$, there exists a $\delta>0$ such that if $\varphi\in\mathcal{H}$ satisfies $\|\varphi-v\|_{\mathcal{H}}<\delta$ then $u_{\varphi}$ is globally defined and $\inf\limits_{z\in Z}\|u_{\varphi}(t)-z\|_{\mathcal{H}}<\varepsilon$ for all $t\in\mathbb{R}$, where $u_{\varphi}(t)$ is the solution to \eqref{eq1.1} corresponding to the initial data in $\mathcal{H}$.
\end{definition}
\begin{theorem}\label{t1.1}
Let $d\geq3,\ \frac{2d-\alpha}{d}<q<\frac{2d-\alpha+2}{d}$,\ $\left\|V_{-}\right\|_{\frac{d}{2}}\leq\mathcal{S}$. There exists a $a_0>0$ such that, for any $a\in(0,a_0)$, $I(u)$ restricted to $S(a)$ has a ground state. This ground state is a (local) minimizer of $I(u)$ in the set
$\mathbf{V}(a)$. In addition, if $(u_n)\subset\mathbf{V} (a)$ is such that $I(u_n)\rightarrow m(a)$, then, up to translation, $u_n \rightarrow u \in \mathcal{M}_a$ in $H^1(\mathbb{R}^d, \mathbb{R})$.
\end{theorem}
\begin{remark}\label{R1.3}
$(1)$ There exists a ground state which is a real valued, positive, radially symmetric decreasing function.
In fact, if $u\in S(a)$ is a ground state then its Schwartz symmetrization is clearly also a ground state.

$(2)$ Under the assumption of Theorem \ref{t1.1}, it can be proved that, for any $a\in(0, a_0)$,
\begin{equation*}
  \mathcal{M}_a=\{e^{i\theta}u \text{ for some } \theta\in\mathbb{R}, u\in S(a)\cap H^1(\mathbb{R}^d,\mathbb{R}), I(u)=m(a)\}
\end{equation*}
by applying the argument of \cite[Remark 1.4]{{LJJJ2022}}.

$(3)$ Theorem \ref{t1.1} generalizes the conclusion in \cite{LJJJ2022}. In addition, this result also holds for biharmonic operators.

$(4)$ Using techniques in \cite{TBAQ2023}, we can obtain $\lambda<0$.
\end{remark}
\begin{theorem}\label{t1.2}
Let $d\geq3,\ \frac{2d-\alpha}{d}<q<\frac{2d-\alpha+2}{d},\ \left\|V_{-}\right\|_{\frac{d}{2}}\leq\mathcal{S}$ and $a_0>0$ be given in Theorem \ref{t1.1}. Assume $V$ satisfies \eqref{eq3.2} and \eqref{eq3.3}, then for any $a\in(0,a_0)$, the set $\mathcal{M}_a$ is compact, up to translation, and it is orbitally stable.
\end{theorem}

In section 2, we first present some basic inequalities and then prove the existence of normalized ground states. In sections 3, we aim to estimate nonlinearity by using new Strichartz estimates and prove Theorem \ref{t1.2}.
\section{Ground state solution}

We shall make use of the following classical inequalities:

\textbf{Sobolev inequality}(see \cite{HB1983}) For any $d \geq 3$ there exists an optimal constant $\mathcal{S}>0$ depending only on $d$, such that
\begin{equation*}
  \mathcal{S}\|u\|_{2^*}^2 \leq\|\nabla u\|_2^2,\ \forall u \in H^1(\mathbb{R}^d).
\end{equation*}

\textbf{Gagliardo-Nirenberg inequality}(see \cite{LN1959}) If $d \geq 2, p \in[2, \frac{2d}{d-2})$ and $\beta=d(\frac{1}{2}-\frac{1}{p})$, then
\begin{equation*}
  \|u\|_p \leq C_{d, p}\|\nabla u\|_2^\beta\|u\|_2^{1-\beta},
\end{equation*}
for all $u \in H^1(\mathbb{R}^d)$.

\textbf{Hardy-Littlewood-Sobolev inequality}(see \cite{ELML2001}) Let $t, r>1$ and $0<\alpha<d$ be such that $\frac{1}{r}+\frac{1}{t}+\frac{\alpha}{d}=2, f \in L^r(\mathbb{R}^d)$ and $h \in L^t(\mathbb{R}^d)$. Then, there exists a constant $C(d, \alpha, r, t)>0$ such that
$$
\left|\int_{\mathbb{R}^d}\left(I_{\alpha}* f\right) h d x\right| \leq C(d, \alpha, r, t)\|f\|_r\|h\|_t .
$$
The equation holds if and only if $t=r=\frac{2 d}{2 d-\alpha}$, then
$$
C(d, \alpha, r, t)=C(d, \alpha)=\pi^{\frac{\alpha}{2}} \frac{\Gamma\left(\frac{d-\alpha}{2}\right)}{\Gamma\left(d-\frac{\alpha}{2}\right)}\left\{\frac{\Gamma\left(\frac{d}{2}\right)}{\Gamma(d)}\right\}^{-1+\frac{\alpha}{d}} 
$$
with $f(x)=Ch(x)$ and
$$
h(x)=A\left(\gamma^2+|x-c|^2\right)^{-\frac{2 d-\alpha}{2}}
$$
for some $A \in \mathbb{C}, 0 \neq \gamma \in \mathbb{R}$ and $c \in \mathbb{R}^d$.

For any $u \in H^1(\mathbb{R}^d)$, take $t=r=\frac{2d}{2d-\alpha}, f=h=|u|^q$ in Hardy-Littlewood-Sobolev inequality, we have
\begin{equation} \label{eq2.1}
\int_{\mathbb{R}^d}\left(I_{\alpha  }*|u|^q\right)|u|^q d x \leq C(d, \alpha)\|u\|_{\frac{2 d q}{2 d-\alpha}}^{2 q} \leq C_{d, q,\alpha  }\left\|\nabla u\right\|_2^{d(q-2)+\alpha}\|u\|_2^{2 q-d(q-2)-\alpha}
\end{equation}
by using Gagliardo-Nirenberg inequality. Now, we define the best constant $S_{\alpha}$ as 
$$
S_{\alpha}:=\inf _{u \in D^{1, 2}(\mathbb{R}^d)\setminus \{0\}} \frac{\|\nabla u\|_2^2}{\left(\int_{\mathbb{R}^d}\left(I_{\alpha}*|u|^{2_{\alpha}^*}\right)|u|^{2_{\alpha}^*} dx\right)^{\frac{1}{2_\alpha^*}}}.
$$
Moreover, the constants $S_\alpha$ and $\mathcal{S}$ have the relationship: $S_\alpha=\mathcal{S}C_{d,\alpha}^{-\frac{1}{2_\alpha^*}}$.

Now, it follows from \eqref{eq2.1} that
\begin{eqnarray*}
I(u)&=&\frac{1}{2}\|\nabla u\|_2^2+\frac{1}{2}\int_{\mathbb{R}^d}V(x)u^2dx -\frac{1}{2  q}\int_{\mathbb{R}^d}(I_{\alpha}\ast|u|^{q})|u|^{q}dx-\frac{1}{2\cdot2_\alpha^*}\int_{\mathbb{R}^d}(I_{\alpha}\ast|u|^{2_{\alpha}^*})|u|^{2_{\alpha}^*}dx \\
&\geq&\frac{1}{2}\left(1-\left\|V_{-}\right\|_{\frac{d}{2}} \mathcal{S}^{-1}\right)  \int_{\mathbb{R}^d}|\nabla u|^2 d x-\frac{C_{d, q,\alpha}  a^{\frac{2 q-d(q-2)-\alpha}{2}}}{2q}\left(\int_{\mathbb{R}^d}|\nabla u|^2 d x\right)^{\frac{d(q-2)+\alpha}{2}}\nonumber\\
&&-\frac{1}{2\cdot2_\alpha^*\cdot S_\alpha^{
 2_\alpha^* }}\left(\int_{\mathbb{R}^d}|\nabla u|^2 d x\right)^{2_\alpha^*},
\end{eqnarray*}
we consider the function $f(a,\rho)$ defined on $(0, \infty)\times(0, \infty)$ by
$$
f(a,\rho):= \frac{1}{2}\left(1-\left\|V_{-}\right\|_{\frac{d}{2}} \mathcal{S}^{-1}\right) -\frac{C_{d, q,\alpha}  a^{\frac{2 q-d(q-2)-\alpha}{2}}}{2q}\rho^{\frac{d(q-2)+\alpha-2}{2}}-\frac{1}{2\cdot2_\alpha^*\cdot S_\alpha^{
 2_\alpha^* }}\rho^{2_\alpha^*-1}.
$$
For each $a\in(0, \infty)$, its restriction $f_a(\rho)$ defined on $(0, \infty)$ by $f_a(\rho):=f(a, \rho)$.
\begin{lemma}\label{L1.1}
For each $a>0$, the function $f(a,\rho)$ has a unique global maximum and the maximum value satisfies
$$
\begin{cases}\max\limits_{\rho>0} f_a(\rho)>0, & \text { if } \quad a<a_0, \\ \max\limits_{\rho>0}f_a(\rho)=0 & \text { if } \quad a=a_0, \\ \max\limits_{\rho>0}f_a(\rho)<0, & \text { if } \quad a>a_0,\end{cases}
$$
where
$$
a_0:=\left[\frac{ \frac{1}{2}\left(1-\left\|V_{-}\right\|_{\frac{d}{2}} \mathcal{S}^{-1}\right)}{2 K}\right]^{\frac{d}{d+2-\alpha}}>0
$$
with
\begin{eqnarray*}
K &:=&\frac{C_{d, q,\alpha}  }{2q}\left[ \frac{2_\alpha^*[2-\alpha-d(q-2)]C_{d,q,\alpha} S_\alpha^{2_\alpha^*}}{2q(2_\alpha^*-1)}\right]^{ \frac{d(q-2)+\alpha-2}{2\cdot2_\alpha^*-d(q-2)-\alpha}} \\ &&+\frac{1}{2\cdot2_\alpha^*\cdot S_\alpha^{2_\alpha^*}}\left[ \frac{2_\alpha^*[2-\alpha-d(q-2)]C_{d,q,\alpha} S_\alpha^{2_\alpha^*}}{2q(2_\alpha^*-1)}\right]^{\frac{2(2_\alpha^*-1)}{2\cdot2_\alpha^*-d(q-2)-\alpha}}\\
&>&0.
\end{eqnarray*}
\end{lemma}
\begin{proof}
By definition of $f_a(\rho)$, we have that
$$
f_a^{\prime}(\rho)=-\frac{C_{d, q,\alpha}  a^{\frac{2 q-d(q-2)-\alpha}{2}}(d(q-2)+\alpha-2)}{4q}\rho^{\frac{d(q-2)+\alpha-4}{2}}-\frac{2_\alpha^*-1}{2}\cdot\frac{1}{2_\alpha^*\cdot S_\alpha^{
 2_\alpha^* }}\rho^{2_\alpha^*-2}.
$$
Hence, the equation $f^{\prime}(\rho)=0$ has a unique solution given by
$$
\rho_a=\left[ \frac{2_\alpha^*[2-\alpha-d(q-2)]C_{d,q,\alpha} S_\alpha^{2_\alpha^*}}{2q(2_\alpha^*-1)}\right]^{\frac{2}{2\cdot2_\alpha^*-d(q-2)-\alpha}} a^{\frac{2q-d(q-2)-\alpha}{2\cdot2_\alpha^*-d(q-2)-\alpha}} .
$$
Taking into account that $f_a(\rho) \rightarrow-\infty$ as $\rho \rightarrow 0$ and $f_a(\rho) \rightarrow-\infty$ as $\rho \rightarrow \infty$, we obtain that $\rho_a$ is the unique global maximum point of $f(a,\rho)$ and the maximum value is
\begin{eqnarray*}
\max_{\rho>0}f_a(\rho) &=&\frac{1}{2}\left(1-\left\|V_{-}\right\|_{\frac{d}{2}} \mathcal{S}^{-1}\right) -\frac{C_{d, q,\alpha}  a^{\frac{2 q-d(q-2)-\alpha}{2}}}{2q}\left[ \frac{2_\alpha^*[2-\alpha-d(q-2)]C_{d,q,\alpha} S_\alpha^{2_\alpha^*}}{2q(2_\alpha^*-1)}\right]^{ \frac{d(q-2)+\alpha-2}{2\cdot2_\alpha^*-d(q-2)-\alpha}}    \\
&&\cdot a^{\frac{2q-d(q-2)-\alpha}{2\cdot2_\alpha^*-d(q-2)-\alpha}\cdot\frac{d(q-2)+\alpha-2}{2}}-\frac{a^{\frac{2q-d(q-2)-\alpha}{2\cdot2_\alpha^*-d(q-2)-\alpha}\cdot (2_\alpha^*-1) }}{2\cdot2_\alpha^*\cdot S_\alpha^{2_\alpha^*}}\left[ \frac{2_\alpha^*[2-\alpha-d(q-2)]C_{d,q,\alpha} S_\alpha^{2_\alpha^*}}{2q(2_\alpha^*-1)}\right]^{\frac{2(2_\alpha^*-1)}{2\cdot2_\alpha^*-d(q-2)-\alpha}}  \\
&=& \frac{1}{2}\left(1-\left\|V_{-}\right\|_{\frac{d}{2}} \mathcal{S}^{-1}\right) -\frac{C_{d, q,\alpha}  }{2q}\left[ \frac{2_\alpha^*[2-\alpha-d(q-2)]C_{d,q,\alpha} S_\alpha^{2_\alpha^*}}{2q(2_\alpha^*-1)}\right]^{ \frac{d(q-2)+\alpha-2}{2\cdot2_\alpha^*-d(q-2)-\alpha}} \cdot a^{\frac{d+2-\alpha}{d}}   \\
&&-\frac{1}{2\cdot2_\alpha^*\cdot S_\alpha^{2_\alpha^*}}\left[ \frac{2_\alpha^*[2-\alpha-d(q-2)]C_{d,q,\alpha} S_\alpha^{2_\alpha^*}}{2q(2_\alpha^*-1)}\right]^{\frac{2(2_\alpha^*-1)}{2\cdot2_\alpha^*-d(q-2)-\alpha}}\cdot a^{\frac{d+2-\alpha}{d}}     \\
&=& \frac{1}{2}\left(1-\left\|V_{-}\right\|_{\frac{d}{2}} \mathcal{S}^{-1}\right)-Ka^{\frac{d+2-\alpha}{d}}.
\end{eqnarray*}
By the definition of $a_0$, we have that $\max\limits_{\rho>0} f_{a_0}(\rho)=0$, and hence the lemma follows.
\end{proof}
\begin{lemma}\label{L1.2}
Let $(a_1, \rho_1) \in(0, \infty) \times(0, \infty)$ be such that $f(a_1, \rho_1) \geq 0$. Then for any $a_2 \in(0, a_1]$, it holds
$$
f(a_2, \rho_2) \geq 0  \text { if }  \rho_2 \in\left[\frac{a_2}{a_1} \rho_1, \rho_1\right].
$$
\end{lemma}
\begin{proof}
Note that $a \rightarrow f(a, \rho)$ is a non-increasing function, so we have
\begin{equation}\label{eq1.5}
  f(a_2, \rho_1) \geq f(a_1, \rho_1) \geq 0.
\end{equation}
Now, by direct calculations, one has
\begin{eqnarray}\label{eq1.6}
f(a_2, \frac{a_2}{a_1} \rho_1)&=&\frac{1}{2}\left(1-\left\|V_{-}\right\|_{\frac{d}{2}} \mathcal{S}^{-1}\right) -\frac{C_{d, q,\alpha}  a_2^{\frac{2 q-d(q-2)-\alpha}{2}}}{2q}\left(\frac{a_2}{a_1} \rho_1\right)^{\frac{d(q-2)+\alpha-2}{2}}\nonumber\\
&&-\frac{1}{2\cdot2_\alpha^*\cdot S_\alpha^{
 2_\alpha^* }}\left(\frac{a_2}{a_1} \rho_1\right)^{2_\alpha^*-1}\nonumber\\
&\geq&\frac{1}{2}\left(1-\left\|V_{-}\right\|_{\frac{d}{2}} \mathcal{S}^{-1}\right) -\frac{C_{d, q,\alpha}  a_1^{\frac{2 q-d(q-2)-\alpha}{2}}}{2q}  \rho_1^{\frac{d(q-2)+\alpha-2}{2}}-\frac{1}{2\cdot2_\alpha^*\cdot S_\alpha^{
 2_\alpha^* }}\rho_1^{2_\alpha^*-1}  \nonumber\\
&=&f(a_1, \rho_1)\nonumber\\
&\geq& 0
\end{eqnarray}
since $q-1>0$. We claim that if $f_{a_2}\left(\rho^{\prime}\right) \geq 0$ and $f_{a_2}\left(\rho^{\prime \prime}\right) \geq 0$, then
\begin{equation}\label{eq1.7}
  f(a_2, \rho)=f_{a_2}(\rho) \geq 0 \text { for any } \rho \in\left[\rho^{\prime}, \rho^{\prime\prime}\right] .
\end{equation}
In fact, if $f_{a_2}(\rho)<0$ for some $\rho \in (\rho^{\prime}, \rho^{\prime \prime} )$, then there exists a local minimum point on $(\rho_1, \rho_2)$ and this contradicts the fact that the function $f_{a_2}(\rho)$ has a unique critical point which has to coincide necessarily with its unique global maximum via Lemma \ref{L1.1}. Now, we can choose $\rho^{\prime}=\frac{a_2}{a_1}\rho_1$ and $\rho^{\prime \prime}=\rho_1$, then \eqref{eq1.5}-\eqref{eq1.7} imply the lemma.
\end{proof}
\begin{lemma}\label{L1.3}
For any $u \in S(a)$, we have that
$$
I(u) \geq\|\nabla u\|_2^2 f\left(a,\|\nabla u\|_2^2\right) .
$$
\end{lemma}
\begin{proof}
According to the definition of $f(a,\rho)$, the lemma clearly holds.
\end{proof}
Note that by Lemma \ref{L1.1} and Lemma \ref{L1.2}, we have that $f(a_0, \rho_0)=0$ and $f(a, \rho_0)>0$ for all $a \in(0, a_0)$. We define
$$
\mathbf{B}_{\rho_0}:=\left\{u \in H^1(\mathbb{R}^d):\|\nabla u\|_2^2<\rho_0\right\} \text { and } \mathbf{V}(a):=S(a) \cap \mathbf{B}_{\rho_0} .
$$
Now, we consider the following local minimization problem: for any $a \in(0, a_0)$,
$$
m(a):=\inf_{u \in \mathbf{V}(a)} I(u).
$$
\begin{lemma}\label{L1.4}
For any $a \in\left(0, a_0\right)$, it holds that
$$
m(a)=\inf _{u \in \mathbf{V}(a)} I(u)<0<\inf _{u \in \partial \mathbf{V}(a)} I(u) .
$$
\end{lemma}
\begin{proof}
For any $u \in \partial \mathbf{V}(a)$, we have $\|\nabla u\|_2^2=\rho_0$. By using Lemma \ref{L1.3}, we get
$$
I(u) \geq\|\nabla u\|_2^2 f(\|u\|_2^2,\|\nabla u\|_2^2)=\rho_0 f(a, \rho_0)>0.
$$
Now let $u \in S(a)$ be arbitrary but fixed. For $s \in(0, \infty)$, define
$$
u_s(x):=s^{\frac{d}{2}} u(s x) .
$$
Obviously, $u_s \in S(a)$ for any $s \in(0, \infty)$. We define on $(0, \infty)$ the map,
\begin{eqnarray*}
\psi_u(s):=I(u_s)&=&\frac{1}{2}\|\nabla u_s\|_2^2+\frac{1}{2}\int_{\mathbb{R}^d}V(x)u_s^2dx -\frac{1}{2 q}\int_{\mathbb{R}^d}(I_{\alpha}\ast|u_s|^{q})|u_s|^{q}dx\\
&&-\frac{1}{2\cdot2_\alpha^*}\int_{\mathbb{R}^d}(I_{\alpha}\ast|u_s|^{2_{\alpha}^*})|u_s|^{2_{\alpha}^*}dx\\
&\leq&\frac{s^2}{2}\left(1+\|V\|_{\frac{d}{2}} \mathcal{S}^{-1}\right)\|\nabla u\|_2^2-\frac{s^{d(q-2)+\alpha}}{2 q}\int_{\mathbb{R}^d}(I_{\alpha}\ast|u_s|^{q})|u_s|^{q}dx\\
&&-\frac{s^{2\cdot2_\alpha^*}}{2\cdot2_\alpha^*}\int_{\mathbb{R}^d}(I_{\alpha}\ast|u_s|^{2_{\alpha}^*})|u_s|^{2_{\alpha}^*}dx.
\end{eqnarray*}
Note that $d(q-2)+\alpha<2$ and $2_\alpha^*>1$, we see that $\psi_u(s) \rightarrow 0^{-}$ as $s \rightarrow 0$. Hence, there exists $s_0>0$ small enough such that $\left\|\nabla(u_{s_0})\right\|_2^2=s_0^2\|\nabla u\|_2^2<\rho_0$ and $I(u_*)=\psi_\mu(s_0)<0$, which implies that $m(a)<0$.
\end{proof}
We now introduce the set
$$
\mathcal{M}_a:=\left\{u \in \mathbf{V}(a): I(u)=m(a)\right\}
$$
and collect some properties of $m(a)$.
\begin{lemma}\label{L1.5}
$\mathrm{(i)}$ $a \in\left(0, a_0\right) \mapsto m(a)$ is a continuous mapping.

$\mathrm{(ii)}$ Let $a \in\left(0, a_0\right)$. We have for all $\alpha \in(0, a): m(a) \leq m(\alpha)+m(a-\alpha)$ and if $m(\alpha)$ or $m(a-\alpha)$ is reached then the inequality is strict.
\end{lemma}
\begin{proof}
(i) Let $a \in(0, a_0)$ be arbitrary and $(a_n) \subset(0, a_0)$ be such that $a_n \rightarrow a$. From the definition of $m\left(a_n\right)$ and since $m\left(a_n\right)<0$(see Lemma \ref{L1.4}), for any $\varepsilon>0$ sufficiently small, there exists $u_n \in \mathbf{V}\left(a_n\right)$ such that
\begin{equation}\label{eq1.8}
  I\left(u_n\right) \leq m\left(a_n\right)+\varepsilon \text { and } I\left(u_n\right)<0 .
\end{equation}
Set $y_n:=\sqrt{\frac{a}{a_n}} u_n$ and hence $y_n \in S(a)$. We have that $y_n \in \mathbf{V}(a)$. In fact, if $a_n \geq a$, then
$$
\left\|\nabla y_n\right\|_2^2=\frac{a}{a_n}\left\|\nabla u_n\right\|_2^2 \leq\left\|\nabla u_n\right\|_2^2<\rho_0 .
$$
If $a_n<a$, we have $f\left(a_n, \rho\right) \geq 0$ for any $\rho \in\left[\frac{a_n}{a} \rho_0, \rho_0\right]$ by Lemma \ref{L1.2}. Hence, it follows from Lemma \ref{L1.3} and \eqref{eq1.8} that $f\left(a_n,\left\|\nabla u_n\right\|_2^2\right)<0$, thus $\left\|\nabla u_n\right\|_2^2<\frac{a_n}{a} \rho_0$ and
$$
\left\|\nabla y_n\right\|_2^2=\frac{a}{a_n}\left\|\nabla u_n\right\|_2^2<\frac{a}{a_n}\cdot\frac{a_n}{a} \rho_0=\rho_0 .
$$
Since $y_n \in \mathbf{V}(a)$, we can write
$$
m(a) \leq I\left(y_n\right)=I\left(u_n\right)+\left[I\left(y_n\right)-I\left(u_n\right)\right],
$$
where
\begin{eqnarray*}
&&I\left(y_n\right)-I\left(u_n\right)\\
&=&\frac{1}{2}\left(\frac{a}{a_n}-1\right)\left\|\nabla u_n\right\|_2^2-\frac{1}{2q}\left[\left(\frac{a}{a_n}\right)^{q}-1\right]\int_{\mathbb{R}^d}(I_{\alpha}\ast|u_n|^{q})|u_n|^{q}dx\\
&&+\frac{1}{2}\left[\frac{a}{a_n}-1\right]\int_{\mathbb{R}^d}V(x)u_n^2dx-\frac{1}{2\cdot2_\alpha^*}\left[\left(\frac{a}{a_n}\right)^{2_\alpha^*}-1\right]\int_{\mathbb{R}^d}(I_{\alpha}\ast|u_n|^{2_{\alpha}^*})|u_n|^{2_{\alpha}^*}dx.
\end{eqnarray*}
Since $\left\|\nabla u_n\right\|_2^2<\rho_0$, then $\int_{\mathbb{R}^d}(I_{\alpha}\ast|u_n|^{q})|u_n|^{q}dx$ and $\int_{\mathbb{R}^d}(I_{\alpha}\ast|u_n|^{2_{\alpha}^*})|u_n|^{2_{\alpha}^*}dx$ are uniformly bounded. Hence, we have
\begin{equation}\label{eq1.9}
  m(a) \leq I\left(y_n\right)=I\left(u_n\right)+o_n(1)
\end{equation}
as $n \rightarrow \infty$. Combining \eqref{eq1.8} and \eqref{eq1.8}, we have
$$
m(a) \leq m\left(a_n\right)+\varepsilon+o_n(1).
$$
Now, let $u \in \mathbf{V}(a)$ be such that
$$
I(u) \leq m(a)+\varepsilon  \text { and }   I(u)<0 .
$$
Set $u_n:=\sqrt{\frac{a_n}{a}} u$ and hence $u_n \in S\left(a_n\right)$. Clearly, $\|\nabla u\|_2^2<\rho_0$ and $a_n \rightarrow a$ imply $\left\|\nabla u_n\right\|_2^2<\rho_0$ for $n$ large enough, so that $u_n \in \mathbf{V}\left(a_n\right)$ and $I\left(u_n\right) \rightarrow I(u)$. So we have
$$
m\left(a_n\right) \leq I\left(u_n\right)=I(u)+\left[I\left(u_n\right)-I(u)\right] \leq m(a)+\varepsilon+o_n(1) .
$$
Therefore, since $\varepsilon>0$ is arbitrary, we deduce that $m\left(a_n\right) \rightarrow m(a)$.

(ii) Fixed $\alpha \in(0, a)$, it is sufficient to prove that the following holds
\begin{equation}\label{eq1.10}
  \forall \theta \in\left(1, \frac{a}{\alpha}\right]: m(\theta \alpha) \leq \theta m(\alpha)
\end{equation}
and if $m(\alpha)$ is reached, the inequality is strict. In fact, if \eqref{eq1.10} holds then we have
\begin{eqnarray*}
m(a) & =&\frac{a-\alpha}{a} m(a)+\frac{\alpha}{a} m(a)\\
&=&\frac{a-\alpha}{a} m\left(\frac{a}{a-\alpha}(a-\alpha)\right)+\frac{\alpha}{a} m\left(\frac{a}{\alpha} \alpha\right) \\
& \leq& m(a-\alpha)+m(\alpha)
\end{eqnarray*}
with a strict inequality if $m(\alpha)$ is reached. In order to obtain that \eqref{eq1.10}, note that in view of Lemma \ref{L1.4}, there exists a $u \in V(\alpha)$ such that
\begin{equation}\label{eq1.11}
I(u) \leq m(\alpha)+\varepsilon \text { and } I(u)<0
\end{equation}
for any $\varepsilon>0$ sufficiently small. By using Lemma \ref{L1.2}, $f(\alpha, \rho) \geq 0$ for any $\rho \in\left[\frac{\alpha}{a} \rho_0, \rho_0\right]$. Hence, it follows from Lemma \ref{L1.3} and \eqref{eq1.11} that
\begin{equation}\label{eq1.12}
\|\nabla u\|_2^2<\frac{\alpha}{a} \rho_0 .
\end{equation}
Now, consider $v=\sqrt{\theta} u$. We first note that $\|v\|_2^2=\theta\|u\|_2^2=\theta \alpha$ and also, $\|\nabla v\|_2^2=\theta\|\nabla u\|_2^2< \rho_0$ because of \eqref{eq1.12}. Thus $v \in \mathbf{V}(\theta \alpha)$ and we can write
\begin{eqnarray*}
m(\theta \alpha)  \leq I(v)&=&\frac{1}{2} \theta\|\nabla u\|_2^2+\frac{1}{2}\theta\int_{\mathbb{R}^d}V(x)u^2dx-\frac{1}{2q} \theta^{q}\int_{\mathbb{R}^d}(I_{\alpha}\ast|u|^{q})|u|^{q}dx\\
&&-\frac{1}{2\cdot2_\alpha^*} \theta^{2_\alpha^*}\int_{\mathbb{R}^d}(I_{\alpha}\ast|u|^{2_{\alpha}^*})|u|^{2_{\alpha}^*}dx \\
& <&\frac{1}{2} \theta\|\nabla u\|_2^2+\frac{1}{2}\theta\int_{\mathbb{R}^d}V(x)u^2dx-\frac{1}{2q} \theta \int_{\mathbb{R}^d}(I_{\alpha}\ast|u|^{q})|u|^{q}dx\\
&&-\frac{1}{2\cdot2_\alpha^*} \theta \int_{\mathbb{R}^d}(I_{\alpha}\ast|u|^{2_{\alpha}^*})|u|^{2_{\alpha}^*}dx \\
&=&\theta I(u) \leq \theta(m(\alpha)+\varepsilon).
\end{eqnarray*}
Since $\varepsilon>0$ is arbitrary, we have that $m(\theta \alpha) \leq \theta m(\alpha)$. If $m(\alpha)$ is reached then we can let $\varepsilon=0$ in \eqref{eq1.11} and thus the strict inequality follows.
\end{proof}
\begin{lemma}\label{L1.6}
Let $\left(v_n\right) \subset \mathbf{B}_{\rho_0}$ be such that $\int_{\mathbb{R}^d}(I_{\alpha}\ast|u|^{q})|u|^{q}dx \rightarrow 0$. Then there exists a $\beta_0>0$ such that
$$
I\left(v_n\right) \geq \beta_0\left\|\nabla v_n\right\|_2^2+o_n(1).
$$
\end{lemma}
\begin{proof}
Indeed, using the Sobolev inequality, we have
\begin{eqnarray*}
I\left(v_n\right) & =&\frac{1}{2}\|\nabla v_n\|_2^2+\frac{1}{2}\int_{\mathbb{R}^d}V(x)v_n^2dx-\frac{1}{2\cdot2_\alpha^*}   \int_{\mathbb{R}^d}(I_{\alpha}\ast|v_n|^{2_{\alpha}^*})|v_n|^{2_{\alpha}^*}dx+o_n(1) \\
& \geq& \frac{1}{2}\left(1-\left\|V_{-}\right\|_{\frac{d}{2}} \mathcal{S}^{-1}\right)\left\|\nabla v_n\right\|_2^2-\frac{1}{2\cdot2_\alpha^*\cdot S_\alpha^{
 2_\alpha^* }}\left\|\nabla v_n\right\|_2^{2\cdot2_\alpha^*}+o_n(1) \\
& =&\left\|\nabla v_n\right\|_2^2\left[\frac{1}{2}\left(1-\left\|V_{-}\right\|_{\frac{d}{2}} \mathcal{S}^{-1}\right)-\frac{1}{2\cdot2_\alpha^*\cdot S_\alpha^{
 2_\alpha^* }}\left\|\nabla v_n\right\|_2^{2\cdot2_\alpha^*-2}\right]+o_n(1) \\
& \geq&\left\|\nabla v_n\right\|_2^2\left[\frac{1}{2}\left(1-\left\|V_{-}\right\|_{\frac{d}{2}} \mathcal{S}^{-1}\right)-\frac{1}{2\cdot2_\alpha^*\cdot S_\alpha^{
 2_\alpha^* }}\rho_0^{2_\alpha^*-1}\right]+o_n(1).
\end{eqnarray*}
Now, since $f\left(a_0, \rho_0\right)=0$, we have that
\begin{equation*}
  \beta_0:=\frac{1}{2}\left(1-\left\|V_{-}\right\|_{\frac{d}{2}} \mathcal{S}^{-1}\right)-\frac{1}{2\cdot2_\alpha^*\cdot S_\alpha^{
 2_\alpha^* }}\rho_0^{2_\alpha^*-1}=\frac{C_{d, q,\alpha}  a_0^{\frac{2 q-d(q-2)-\alpha}{2}}}{2q}\rho_0^{\frac{d(q-2)+\alpha-2}{2}}>0,
\end{equation*}
which completes the proof.
\end{proof}
\begin{lemma}\label{L1.7}
For any $a \in\left(0, \varepsilon_0\right)$, let $\left(u_n\right) \subset \mathbf{B}_{\rho_0}$ be such that $\left\|u_n\right\|_2^2 \rightarrow a$ and $I\left(u_n\right) \rightarrow m(a)$. Then, there exist a $\beta_1>0$ and a sequence $\left(y_n\right) \subset \mathbb{R}^d$ such that
\begin{equation}\label{eq1.13}
  \int_{B\left(y_n, R\right)}\left|u_n\right|^2 d x \geq \beta_1>0  \text { for some } R>0.
\end{equation}
\end{lemma}
\begin{proof}
By contradiction, assume that \eqref{eq1.13} does not hold. Since ( $u_n$ ) $\subset B_{\rho_0}$ and $\left\|u_n\right\|_2^2 \rightarrow a$, the sequence $\left(u_n\right)$ is bounded in $H^1(\mathbb{R}^d)$. From Lemma I.1 in \cite{PLL1984} and $2<q<2^*$, we deduce that $\int_{\mathbb{R}^d}(I_{\alpha}\ast|u|^{q})|u|^{q}dx \rightarrow 0$ as $n \rightarrow \infty$. Hence, Lemma \ref{L1.6} implies that $I\left(u_n\right) \geq o_n(1)$. This contradicts the fact that $m(a)<0$ and the lemma follows.
\end{proof}
\begin{lemma}\label{L1.8}
For any $a \in\left(0, a_0\right)$, if $\left(u_n\right) \subset \mathbf{B}_{\rho_0}$ is such that $\left\|u_n\right\|_2^2 \rightarrow a$ and $I\left(u_n\right) \rightarrow m(a)$ then, up to translation, $u_{n} \xrightarrow{H^1(\mathbb{R}^d)} u \in \mathcal{M}_a$. In particular the set $\mathcal{M}_c$ is compact in $H^1(\mathbb{R}^d), u_\rho$ to translation.
\end{lemma}
\begin{proof}
It follows from Lemma \ref{L1.7} and Rellich compactness theorem that there exists a sequence $\left(y_n\right) \subset \mathbb{R}^d$ such that
$$
u_n\left(x-y_n\right) \rightharpoonup u_a \neq 0 \text { in } H^1(\mathbb{R}^d).
$$
Our aim is to prove that $w_n(x):=u_n\left(x-y_n\right)-u_a(x) \rightarrow 0$ in $H^1(\mathbb{R}^d)$. Clearly,
\begin{eqnarray*}
\left\|u_n\right\|_2^2 & =&\left\|u_n\left(x-y_n\right)\right\|_2^2\\
&=&\left\|u_n\left(x-y_n\right)-u_a(x)\right\|_2^2+\left\|u_a\right\|_2^2+o_n(1) \\
& =&\left\|w_n\right\|_2^2+\left\|u_a\right\|_2^2+o_n(1).
\end{eqnarray*}
Hence,
\begin{equation}\label{eq1.14}
  \left\|w_n\right\|_2^2=\left\|u_n\right\|_2^2-\left\|u_a\right\|_2^2+o_n(1)=a-\left\|u_a\right\|_2^2+o_n(1).
\end{equation}
By a similar argument,
\begin{equation}\label{eq1.15}
\left\|\nabla w_n\right\|_2^2=\left\|\nabla u_n\right\|_2^2-\left\|\nabla u_a\right\|_2^2+o_n(1).
\end{equation}
By the translational invariance, it holds
\begin{equation}\label{eq1.16}
I\left(u_n\right)=I\left(u_n\left(x-y_n\right)\right)=I\left(w_n\right)+I\left(u_a\right)+o_n(1) .
\end{equation}
Next, we claim that
$$
\left\|w_n\right\|_2^2 \rightarrow 0.
$$
In fact, denote $a_1:=\left\|u_a\right\|_2^2>0$. By \eqref{eq1.14}, if we show that $a_1=a$ then the claim follows. By contradiction, we assume that $a_1<a$. Using \eqref{eq1.14} and \eqref{eq1.15}, we have $\left\|w_n\right\|_2^2 \leq a$ and $\left\|\nabla w_n\right\|_2^2 \leq\left\|\nabla u_n\right\|_2^2<\rho_0$ for $n$ large enough. Thus, we obtain that $w_n \in \mathbf{V}(\left\|w_n\right\|_2^2)$ and $I\left(w_n\right) \geq m\left(\left\|w_n\right\|_2^2\right)$. Recording that $I\left(u_n\right) \rightarrow m(a)$, in view of \eqref{eq1.16}, we have
$$
m(a)=I\left(w_n\right)+I\left(u_a\right)+o_n(1) \geq m\left(\left\|w_n\right\|_2^2\right)+I\left(u_a\right)+o_n(1).
$$
Since the map $a \mapsto m(a)$ is continuous, we deduce that
\begin{equation}\label{eq1.17}
m(a) \geq m\left(a-a_1\right)+I\left(u_a\right).
\end{equation}
We also have that $u_a \in \mathbf{V}\left(a_1\right)$ by the weak limit. This implies that $I\left(u_a\right) \geq m\left(a_1\right)$. If $I\left(u_a\right)>m\left(a_1\right)$, then it follows from \eqref{eq1.17} and Lemma \ref{L1.5}(ii) that
$$
m(a)>m\left(a-a_1\right)+m\left(a_1\right) \geq m\left(a-a_1+a_1\right)=m(a),
$$
which is impossible. Hence, we have $I\left(u_a\right)=m\left(a_1\right)$, namely $u_a$ is a local minimizer on $\mathbf{V}\left(a_1\right)$. So, using Lemma \ref{L1.5}(ii) with the strict inequality, we deduce from \eqref{eq1.17} that
$$
m(a) \geq m\left(a-a_1\right)+I\left(u_a\right)=m\left(a-a_1\right)+m\left(a_1\right)>m\left(a-a_1+a_1\right)=m(a),
$$
which is impossible. Thus, the claim follows and from \eqref{eq1.14} we obtain that $\left\|u_a\right\|_2^2=a$.

Now, in order to complete the proof, we need to show that $\left\|\nabla w_n\right\|_2^2 \rightarrow 0$. Note that, by using \eqref{eq1.15} and $u_a \neq 0$, we have $\left\|\nabla w_n\right\|_2^2 \leq\left\|\nabla u_n\right\|_2^2<\rho_0$ for $n$ large enough, so $\left(w_n\right) \subset B_{\rho_0}$ and it is bounded in $H^1(\mathbb{R}^d)$. Then by using the Gagliardo-Nirenberg inequality and recalling $\left\|w_n\right\|_2^2 \rightarrow 0$, we also have $\left\|w_n\right\|_q^q \rightarrow 0$. Thus, it follows from Lemma \ref{L1.6} that
\begin{equation}\label{eq1.18}
  I\left(w_n\right) \geq \beta_0\left\|\nabla w_n\right\|_2^2+o_n(1) \text { where } \beta_0>0 .
\end{equation}
Now we remember that
$$
I\left(u_n\right)=I\left(u_a\right)+I\left(w_n\right)+o_n(1) \rightarrow m(a) .
$$
Since $u_a \in \mathbf{V}(a)$ by weak limit, we have that $I\left(u_a\right) \geq m(a)$ and hence $I\left(w_n\right) \leq o_n(1)$. Then we get that $\left\|\nabla w_n\right\|_2^2 \rightarrow 0$ by using \eqref{eq1.18}.
\end{proof}
\begin{proof}[\bf Proof of Theorem \ref{t1.1}]
The fact that if $\left(u_n\right) \subset V(a)$ is such that $I\left(u_n\right) \rightarrow m(a)$ then, up to translation, $u_n \rightarrow u \in M_a$ in $H^1(\mathbb{R}^d)$ follows from Lemma \ref{L1.8}. In particular, it insures the existence of a minimizer for $I(u)$ on $\mathbf{V}(a)$ and this minimizer is a ground state.
\end{proof}
\section{Orbital stability}
In this section we focus on the local existence of solutions to the following Cauchy problem
\begin{equation}\label{eq3.1}
  \left\{\begin{array}{l}
i \partial_t u+\Delta u -V(x)u+(I_{\alpha}\ast|u|^{q})|u|^{q-2}u+(I_{\alpha}\ast|u|^{2_{\alpha}^*})|u|^{2_{\alpha}^*-2}u=0,  \\
\left.u\right|_{t=0}=\varphi \in  H ^1(\mathbb{R}^d).
\end{array}\right.
\end{equation}
Next we give the notion of integral equation associated with \eqref{eq3.1}. In order to do that first we give another definition.
\begin{definition}\label{D3.1}
If $d \geq 3$ the pair $(m, n)$ is said to be admissible if
$$
\frac{2}{m}+\frac{d}{n}=\frac{d}{2},\ p, r \in[2, \infty].
$$
\end{definition}
Note that, we will work with two particular admissible pairs 
$$
\left(m_1, n_1\right):=\left(2q, \frac{2d q}{dq-2}\right) 
$$
and
$$
\left(m_2, n_2\right):=\left(2\cdot2_\alpha^*, \frac{2d\cdot2_\alpha^*}{d\cdot2_\alpha^*-2}\right) .
$$
Now, we introduce the spaces $\mathbf{Y}_T:=L_T^{m_1}L_x^{n_1}  \cap L_T^{m_2}L_x^{n_2}$ and $\mathbf{X}_T:=L_T^{m_1}W_x^{1,n_1} \cap L_T^{m_2}W_x^{1,n_2}$ equipped with the following norms:
$$
\|w\|_{\mathbf{Y}_T}=\|w\|_{L_T^{m_1}L_x^{n_1}}+\|w\|_{L_T^{m_2}L_x^{n_2}} \text { and } \|w\|_{\mathbf{X}_T}=\|w\|_{L_T^{m_1}W_x^{1,n_1}}+\|w\|_{L_T^{m_2}W_x^{1,n_2}},
$$
where $w(t, x)$ defined on $[0, T) \times \mathbb{R}^d$, we have defined
\begin{eqnarray*}
&& \|w(t, x)\|_{L_T^{m}L_x^{n}}=\left(\int_0^T\|w(t, \cdot)\|_{L^n}^m d t\right)^{\frac{1}{m}}, \\
&& \|w(t, x)\|_{L_T^{m}W_x^{1,n}}=\left(\int_0^T\|w(t, \cdot)\|_{W^{1, n}\left(\mathbb{R}^d\right)}^m d t\right)^{\frac{1}{m}} .
\end{eqnarray*}
The potential $V: \mathbb{R}^d \rightarrow \mathbb{R}$ is assumed to satisfy the following assumptions
\begin{equation}\label{eq3.2}
V \in \mathcal{K} \cap L^{\frac{d}{2}}
\end{equation}
and
\begin{equation}\label{eq3.3}
\left\|V_{-}\right\|_{\mathcal{K}}<d(d-2)\alpha(d),
\end{equation}
where $\alpha(d)$ denotes the volume of the unit ball in $\mathbb{R}^d(d\geq3)$, $\mathcal{K}$ is a class of Kato potentials with
$$
\|V\|_{\mathcal{K}}:=\sup\limits_{x \in \mathbb{R}^d} \int_{\mathbb{R}^d} \frac{|V(y)|}{|x-y|} d y
$$
and $V_{-}(x):=\min \{V(x), 0\}$ is the negative part of $V$. Conditions \eqref{eq3.2} and \eqref{eq3.3} are the key to obtain Strichartz estimates, see \cite{JWZY2024}. We defined the operator $\mathcal{L}$ as $\Delta-V$.
\begin{definition}\label{D3.2}
Let $T>0$. We say that $u(t, x)$ is an integral solution of the Cauchy problem \eqref{eq3.1} on the time interval $[0, T]$ if:

$(1)$ $u \in \mathcal{C}([0, T], H^1(\mathbb{R})) \cap \mathbf{X}_T$;

$(2)$ for all $t \in(0, T)$ it holds
 \begin{equation*}
   u(t)=e^{i t  \mathcal{L}} \varphi+i \int_0^t e^{i(t-s)  \mathcal{L}} \left((I_{\alpha}\ast|u|^{q})|u|^{q-2}u+(I_{\alpha}\ast|u|^{2_{\alpha}^*})|u|^{2_{\alpha}^*-2}u\right) d s.
 \end{equation*}
\end{definition}
In order to obtain the local existence result, we recall Strichartz estimates with potential.
\begin{proposition}(see \cite{{JWZY2024}})\label{P3.1}
 Let $d \geq 3$ then for every admissible pairs $(m, n)$ and $(\tilde{m}, \tilde{n})$, there exists a constant $C>0$ such that for every $T>0$, we have

$\mathrm{(i)}$ For every $\varphi \in L^2(\mathbb{R}^d)$, the function $t \mapsto e^{i t \mathcal{L}} \varphi$ belongs to $L_T^{m}L_x^{n} \cap \mathcal{C}([0, T], L^2(\mathbb{R}^d))$ and
\begin{equation}\label{eq3.4}
 \left\|e^{i t \mathcal{L}} \varphi\right\|_{L_T^{m}L_x^{n}} \leq C\|\varphi\|_2 .
\end{equation}

$\mathrm{(ii)}$ Let $F \in L_T^{\tilde{m}'}L_x^{\tilde{n}'}$, where we use a prime to denote conjugate indices. Then the function
$$
t \mapsto \Phi_F(t):=\int_0^t e^{i(t-s) \mathcal{L}} F(s) d s
$$
belongs to $L_T^{m}L_x^{n} \cap \mathcal{C}([0, T], L^2(\mathbb{R}^d))$ and
\begin{equation}\label{eq3.5}
\left\|\Phi_F\right\|_{L_T^{m}L_x^{n}} \leq C\|F\|_{L_T^{\tilde{m}'}L_x^{\tilde{n}'}} .
\end{equation}

$\mathrm{(iii)}$ For every $\varphi \in H^1(\mathbb{R}^d)$, the function $t \mapsto e^{i t \mathcal{L}} \varphi$ belongs to $L_T^{m}W_x^{1,n} \cap \mathcal{C}([0, T], H^1(\mathbb{R}^d))$ and
\begin{equation}\label{eq3.6}
\left\|e^{i t  \mathcal{L}} \varphi\right\|_{L_T^{m}W_x^{1,n}} \leq C\|\varphi\|_{H^1(\mathbb{R}^d)}.
\end{equation}
\end{proposition}

\begin{proposition}\label{P3.2}(\cite{JWZY2024})
If  $V \in \mathcal{K} \cap L^{\frac{d}{2}}$ and $\left\|V_{-}\right\|_{\mathcal{K}}<d(d-2)\alpha(d)$, then
$$
\left\|\mathcal{L}^{\frac{s}{2}} f\right\|_{L^r} \sim\|f\|_{\dot{W}^{s, r}}, \quad\left\|(1+\mathcal{L})^{\frac{s}{2}} f\right\|_{L^r} \sim\|f\|_{W^{s, r}},
$$
where $1<r<\frac{d}{s}$ and $0 \leq s \leq 2$.
\end{proposition}

\begin{proposition}\label{P3.3}
Let $0<\alpha<d$ and $1<s, r, q<\infty$ be such that $\frac{1}{q}+\frac{1}{r}+\frac{1}{s}+\frac{\alpha}{d}=2$. Then,
$$
\left\|\left(I_\alpha * f\right) g\right\|_{r^{\prime}} \leq C(d, s, \alpha)\|f\|_s\|g\|_q, \quad \forall f \in L^s(\mathbb{R}^d), \forall g \in L^q(\mathbb{R}^d).
$$
\end{proposition}
\begin{proof}
Using Hardy-Littlewood-Sobolev inequality for $\frac{1}{t}+\frac{1}{s}+\frac{\alpha}{d}=2$,
\begin{eqnarray*}
\left\|\left(I_\alpha * f\right) g\right\|_{r^{\prime}} & =&\sup _{\|v\|_r=1}\left|\int_{\mathbb{R}^d \times \mathbb{R}^d} I_\alpha(x-y) f(y) g(x) v(x) d y d x\right| \\
& =&C_\alpha \sup _{\|v\|_r=1}\left|\int_{\mathbb{R}^d \times \mathbb{R}^d} \frac{f(y) g(x) v(x)}{|x-y|^{d-\alpha}} d y d x\right| \\
& \leq& C(d, s, \alpha)\|f\|_s\|g v\|_t .
\end{eqnarray*}
By H\"older inequality, we have
 \begin{equation*}
   \left\|\left(I_\alpha * f\right) g\right\|_{r^{\prime}}  \leq C(d, s, \alpha)\|f\|_s\|g v\|_t
 \leq C(d, s, \alpha)\|f\|_s\|g\|_q\|v\|_r
  \leq C(d, s, \alpha)\|f\|_s\|g\|_q
 \end{equation*}
for $\frac{1}{q}+\frac{1}{r}+\frac{1}{s}+\frac{\alpha}{d}=2$.
\end{proof}

The following result will be useful in the sequel.

\begin{lemma}\label{L3.1}
Let $d \geq 3$ and $\frac{2d-\alpha}{d}<\beta \leq 2_\alpha^*$ be given. Then the couple $(m, n)$ defined as follows
$$
m:=2\beta \quad \text { and } \quad n:=\frac{2d\beta}{d\beta-2}
$$
is admissible. Moreover for every admissible couple $(\tilde{m}, \tilde{n})$ there exists a constant $C>0$ such that for every $T>0$ the following inequalities hold:
\begin{equation}\label{eq3.7}
\left\|\int_0^t e^{i(t-s)  \mathcal{L}}(1+\mathcal{L})^{\frac{1}{2}}(I_{\alpha}\ast|u|^{\beta})|u|^{\beta-2}u d s\right\|_{L_T^{\tilde{m}}L_x^{\tilde{n}}}  \leq C   \|u\|_{L_{T}^{m}W_x^{1,n}}^{2\beta-1},
\end{equation}
\begin{eqnarray}\label{eq3.8}
&&\left\|\int_0^t e^{i(t-s)  \mathcal{L}}(1+\mathcal{L})^{\frac{1}{2}}\left[(I_{\alpha}\ast|u|^{\beta})|u|^{\beta-2}u-(I_{\alpha}\ast|v|^{\beta})|v|^{\beta-2}v\right] d s\right\|_{L_T^{\tilde{m}}L_x^{\tilde{n}}} \nonumber\\
&\leq& C \left(\|u\|_{L_{T}^{m}W_x^{1,n}}^{2(\beta-1)}+\|v\|_{L_{T}^{m}W_x^{1,n}}^{2(\beta-1)}\right)\cdot \|u-v\|_{L_{T}^{m}W_x^{1,n}},
\end{eqnarray}
\end{lemma}
\begin{proof}
By direct calculations, one can check that $(m,n)$ is admissible. By using Strichartz estimates, Proposition \ref{P3.2} and Proposition \ref{P3.3}, it follows that
 \begin{eqnarray*}
&&\left\|\int_0^t e^{i(t-s)  \mathcal{L}}(1+\mathcal{L})^{\frac{1}{2}}(I_{\alpha}\ast|u|^{\beta})|u|^{\beta-2}u d s\right\|_{L_T^{\tilde{m}}L_x^{\tilde{n}}}\\
&\leq& C\left\| (1+\mathcal{L})^{\frac{1}{2}}(I_{\alpha}\ast|u|^{\beta})|u|^{\beta-2}u \right\|_{L_T^{m'}L_x^{n'}}\\
&\leq& C\left\| (I_{\alpha}\ast|u|^{q})|u|^{\beta-2}u \right\|_{L_{T}^{ m'}W_x^{1,n'}}\\
&\leq& C \left\| (I_{\alpha}\ast|u|^{\beta})|u|^{\beta-2}u \right\|_{L_{T}^{ m'}L_x^{n'}}+C\left\| (I_{\alpha}\ast|u|^{\beta})|u|^{\beta-2}\nabla u \right\|_{L_{T}^{ m'}L_x^{n'}} +C\left\| (I_{\alpha}\ast|u|^{\beta-1}\nabla u)|u|^{\beta-2} u \right\|_{L_{T}^{ m'}L_x^{n'}}\\
&\leq&C \left\| \|u^{\beta}\|_{L_x^{\frac{4d(\beta-1)}{(3d-2\alpha)\beta+2}}}\|u^{\beta-2}\|_{L_x^{\frac{4d\beta(\beta-1)}{[(3d-2\alpha)\beta+2](\beta-2)} }}\|u\|_{L_x^n} \right\|_{L_{T}^{ m'} }\\
&&+C\left\| \|u^{\beta}\|_{L_x^{\frac{4d(\beta-1)}{(3d-2\alpha)\beta+2}}}\|u^{\beta-2}\|_{L_x^{\frac{4d\beta(\beta-1)}{[(3d-2\alpha)\beta+2](\beta-2)}}}\|\nabla u\|_{L_x^n} \right\|_{L_{T}^{ m'} }\\
&&+C\left\| \|u^{\beta-1}\|_{L_x^{\frac{4d\beta}{(3d-2\alpha)\beta+2}}}\|\nabla u\|_{L_x^n}\|u^{\beta-2}u\|_{L_x^{\frac{4d\beta}{(3d-2\alpha)\beta+2}}} \right\|_{L_{T}^{ m'} }\\
&=&C\left\| \|u\|_{L_x^{\frac{4d\beta(\beta-1)}{(3d-2\alpha)\beta+2}}}^{2(\beta-1)} \|u\|_{L_x^n}\right\|_{L_{T}^{ m'} }+C\left\| \|u\|_{L_x^{\frac{4d\beta(\beta-1)}{(3d-2\alpha)\beta+2}}}^{2(\beta-1)} \|\nabla u\|_{L_x^n}\right\|_{L_{T}^{ m'} }+C\left\| \|u\|_{L_x^{\frac{4d\beta(\beta-1)}{(3d-2\alpha)\beta+2}}}^{2(\beta-1)} \|u\|_{L_x^n}\right\|_{L_{T}^{ m'} }\\
&\leq&C\left\| \|u\|_{L_x^{\frac{4d\beta(\beta-1)}{(3d-2\alpha)\beta+2}}}^{2(\beta-1)} \|u\|_{W_x^{1,n}}\right\|_{L_{T}^{ m'} }+C\left\| \|u\|_{L_x^{\frac{4d\beta(\beta-1)}{(3d-2\alpha)\beta+2}}}^{2(\beta-1)} \|  u\|_{W_x^{1,n}}\right\|_{L_{T}^{ m'} }+C\left\| \|u\|_{L_x^{\frac{4d\beta(\beta-1)}{(3d-2\alpha)\beta+2}}}^{2(\beta-1)} \|u\|_{W_x^{1,n}}\right\|_{L_{T}^{ m'} }\\
&\leq&C  \|u\|_{L_{T}^{ 2\beta}L_x^{\frac{4d\beta(\beta-1)}{(3d-2\alpha)\beta+2}}}^{2(\beta-1)}\cdot \|u\|_{L_{T}^{2\beta}W_x^{1,n}}\leq C \cdot \|u\|_{L_{T}^{2\beta}W_x^{1,n}}^{2\beta-1},
\end{eqnarray*}
which implies \eqref{eq3.7} holds. Similarly, we have
 \begin{eqnarray*}
&&\left\|\int_0^t e^{i(t-s)  \mathcal{L}}(1+\mathcal{L})^{\frac{1}{2}}\left[(I_{\alpha}\ast|u|^{\beta})|u|^{\beta-2}u-(I_{\alpha}\ast|v|^{\beta})|v|^{\beta-2}v\right] d s\right\|_{L_T^{\tilde{m}}L_x^{\tilde{n}}}\\
&\leq& C\left\| (1+\mathcal{L})^{\frac{1}{2}}\left[(I_{\alpha}\ast|u|^{\beta})|u|^{\beta-2}u-(I_{\alpha}\ast|v|^{\beta})|v|^{\beta-2}v\right] \right\|_{L_T^{m'}L_x^{n'}}\\
&\leq& C\left\|\left[(I_{\alpha}\ast|u|^{\beta})|u|^{\beta-2}u-(I_{\alpha}\ast|v|^{\beta})|v|^{\beta-2}v\right]\right\|_{L_{T}^{ m'}W_x^{1,n'}}\\
&\leq& C\left\|(I_{\alpha}\ast(|u|^{\beta}-|v|^{\beta}))|u|^{\beta-2}u\right\|_{L_{T}^{ m'}W_x^{1,n'}}+\left\|(I_{\alpha}\ast|v|^{\beta})(|v|^{\beta-2}v-|u|^{\beta-2}u)\right\|_{L_{T}^{ m'}W_x^{1,n'}}\\
&\leq& C\left\|(I_{\alpha}\ast(|u-v|(|u|^{\beta-1}+|v|^{\beta-1}))|u|^{\beta-2}u\right\|_{L_{T}^{ m'}W_x^{1,n'}}+C\left\|(I_{\alpha}\ast|v|^{\beta})|v-u|(|v|^{\beta-2}+|u|^{\beta-2})\right\|_{L_{T}^{ m'}W_x^{1,n'}}\\
&\leq&C  \left(\|u\|_{L_{T}^{ 2\beta}L_x^{\frac{4d\beta(\beta-1)}{(3d-2\alpha)\beta+2}}}^{2(\beta-1)}+\|v\|_{L_{T}^{ 2\beta}L_x^{\frac{4d\beta(\beta-1)}{(3d-2\alpha)\beta+2}}}^{2(\beta-1)}\right)\cdot \|u-v\|_{L_{T}^{2\beta}W_x^{1,n}}\\
&\leq& C \left(\|u\|_{L_{T}^{2\beta}W_x^{1,n}}^{2(\beta-1)}+\|v\|_{L_{T}^{2\beta}W_x^{1,n}}^{2(\beta-1)}\right)\cdot \|u-v\|_{L_{T}^{2\beta}W_x^{1,n}}
\end{eqnarray*}
since $W_x^{1,n}  \hookrightarrow L_x^{\frac{4d\beta(\beta-1)}{(3d-2\alpha)\beta+2}}$. This completes the proof of the lemma.
\end{proof}

\begin{lemma}\label{L3.2}
For all $1<m, n<\infty$, $L_T^mW_x^{1,n}$ is a separable reflexive Banach space.
\end{lemma}
\begin{proof}
This is a direct consequence of Phillips' theorem, see \cite{{JDJJ1977},{LJJJ2022}}.
\end{proof}
\begin{lemma}\label{L3.3}
For all $R, T>0$ the metric space $\left(\mathbf{B}_{R, T}, d\right)$, where
$$
\mathbf{B}_{R, T}:=\left\{u \in \mathbf{Y}_T:\|(1+\mathcal{L})^{\frac{1}{2}}u\|_{\mathbf{Y}_T} \leq R\right\}
$$
and
$$
d(u, v):=\|(1+\mathcal{L})^{\frac{1}{2}}(u-v)\|_{\mathbf{Y}_T}
$$
is complete.
\end{lemma}
\begin{proof}
Let $\left(u_n\right)$ be a Cauchy sequence. Since $\mathbf{Y}_T$ is a Banach space, there exists $u \in \mathbf{Y}_T$ such that
$$
\lim\limits_{n \rightarrow \infty}\left\|u_n-u\right\|_{\mathbf{Y}_T}=0.
$$
It remains to show that $u \in \mathbf{B}_{R, T}$. By taking a subsequence, we can assume that $l_1:=\lim\limits_{n \rightarrow \infty}\left\|u_n\right\|_{L_T^{m_1}W_x^{1,n_1}}$ and $l_2:=\lim\limits_{n \rightarrow \infty}\left\|u_n\right\|_{L_T^{m_2}W_x^{1,n_2}}$ exist. By Lemma \ref{L3.2}, there exists a subsequence of $\left(u_n\right)$ which converges weakly in $X_{p_1, r_1, T}$. In particular, this sequence converges in the sense of distributions and hence the limit equals $u$. Hence, $\|u\|_{L_T^{m_1}W_x^{1,n_1}} \leq l_1$. Similarly, we also have $\|u\|_{L_T^{m_2}W_x^{1,n_2}} \leq l_2$. Taking the sum, we get $\|u\|_{\mathbf{X}_T} \leq l_1+l_2 \leq R$.
\end{proof}
\begin{lemma}\label{L3.4}
There exists $\rho_0>0$ such that if $\varphi \in H^1(\mathbb{R}^d)$ and $T \in(0,1]$ satisfy
$$
\left\|e^{i t \mathcal{L}} \varphi\right\|_{\mathbf{X}_T} \leq \rho_0,
$$
then there exists a unique integral solution $u(t, x)$ to \eqref{eq3.1} on the time interval $[0, T]$. Moreover $u(t, x) \in L_T^{m}W_x^{1,n}$ for every admissible couple $(m, n)$ and satisfies the following conservation laws
$$
I(u(t))=I(\varphi), \|u(t)\|_2=\|\varphi\|_2  \text { for all } t \in[0, T] .
$$
\end{lemma}
\begin{proof}
\textbf{Step 1.} For any $u \in \mathbf{X}_T$ and $t \in[0, T]$, we define
\begin{equation}\label{eq3.9}
  \Phi(u)(t):=e^{i t  \mathcal{L}} \varphi+i \int_0^t e^{i(t-s)  \mathcal{L}} \left((I_{\alpha}\ast|u|^{q})|u|^{q-2}u+(I_{\alpha}\ast|u|^{2_{\alpha}^*})|u|^{2_{\alpha}^*-2}u\right) d s .
\end{equation}
Next, we will prove that, if $\rho_0>0$ is small enough, then $\Phi$ defines a contraction on the metric space $\left(\mathbf{B}_{2 \rho_0, T}, d\right)$ (see Lemma \ref{L3.3}). In fact, let $u \in \mathbf{B}_{2 \rho_0, T}$ and consider any admissible pair $(\tilde{m}, \tilde{n})$. Let $T \in(0,1]$, it follows from Lemma \ref{L3.2}, \eqref{eq3.7} and \eqref{eq3.9} that for all $u \in \mathbf{B}_{2 \rho_0, T}$,
\begin{equation*}
  \|(1+\mathcal{L})^{\frac{1}{2}}\Phi(u)(t)\|_{L_T^{\tilde{m}}L_x^{\tilde{n}}}\leq C   \left(\|u\|_{L_{T}^{m_1}W_x^{1,n_1}}^{2q-1}+C   \|u\|_{L_{T}^{m_2}W_x^{1,n_2}}^{2\cdot2_\alpha^*-1}\right).
\end{equation*}
In particular if we choose $(\tilde{m}, \tilde{n})=\left(m_1, n_1\right)$ and $(\tilde{m}, \tilde{n})=\left(m_2, n_2\right)$ then
$$
\|\Phi(u)\|_{\mathbf{X}_T} \leq \rho_0+C 2^q \rho_0^{q-1}
$$
and hence if $\gamma_0>0$ is small enough such that $C 2^{q+2} \rho_0^{q-1} \leq \rho_0$, then $B_{2 \rho_0, T}$ is an invariant set of $\Phi$.

Now, let $u, v \in \mathbf{B}_{2 \rho_0, T}$. It follows from \eqref{eq3.8} that, for every admissible pair $(\tilde{m}, \tilde{n})$,
\begin{eqnarray*}
\|(1+\mathcal{L})^{\frac{1}{2}}[\Phi(u)-\Phi(v)]\|_{L_T^{\tilde{m}}L_x^{\tilde{n}}} & \leq& C \left(\|u\|_{L_{T}^{m_1}W_x^{1,n_1}}^{2(q-1)}+\|v\|_{L_{T}^{m_1}W_x^{1,n_1}}^{2(q-1)}\right)\cdot \|u-v\|_{L_{T}^{m_1}W_x^{1,n_1}} \\
&& +C \left(\|u\|_{L_{T}^{m_2}W_x^{1,n_2}}^{2(2_\alpha^*-1)}+\|v\|_{L_{T}^{m_2}W_x^{1,n_2}}^{2(2_\alpha^*-1)}\right)\cdot \|u-v\|_{L_{T}^{m_2}W_x^{1,n_2}} \\
& \leq& C 2^q \rho_0^{q-2}\left(\|u-v\|_{L_{T}^{m_1}W_x^{1,n_1}}+\|u-v\|_{L_{T}^{m_2}W_x^{1,n_2}}\right) .
\end{eqnarray*}
In particular if we choose $(\tilde{m}, \tilde{n})=\left(m_1, n_1\right)$ and $(\tilde{m}, \tilde{n})=\left(m_2, n_2\right)$ then
$$
\|(1+\mathcal{L})^{\frac{1}{2}}[\Phi(u)-\Phi(v)]\|_{\mathbf{Y}_T} \leq C 2^{q+1} \rho_0^{q-2}\|(1+\mathcal{L})^{\frac{1}{2}}(u-v)\|_{\mathbf{Y}_T}
$$
and if we choose $\gamma_0>0$ small enough such that $C 2^{q+1} \rho_0^{q-2}<\frac{1}{2}$ then $\Phi$ is a contraction on $\left(\mathbf{B}_{2 \gamma_0, T}, d\right)$. In particular $\Phi$ has one unique fixed point in this space.

\textbf{Step 2.} The proof of conservation laws is classical, so it is omitted here. Readers can refer to the Proposition 1 and Proposition 2 in \cite{TO2006}.
\end{proof}
Next, we prove that $T_{\varphi}^{\text {max }}=\infty$.
\begin{lemma}\label{L3.5}
Let $\mathcal{K} \subset H^1(\mathbb{R}^d) \backslash\{0\}$ be compact up to translation and assume that $(m, n)$ is an admissible pair with $m \neq \infty$. Then, for every $\rho>0$ there exists $\varepsilon=\varepsilon(\rho)>0$ and $T=T(\rho)>0$ such that
$$
\sup _{\{\varphi \in H^1(\mathbb{R}^d) \mid \mathop{\operatorname{dist}}\limits_{H^1(\mathbb{R}^d)}(\varphi, \mathcal{K})<\varepsilon\}}\left\|e^{i t  \mathcal{L}} \varphi\right\|_{L_T^mW_x^{1,n}}<\rho .
$$
\end{lemma}
\begin{proof}
Firstly, we prove  the existence of a $T>0$ such that
 \begin{equation}\label{eq3.10}
   \sup\limits_{\varphi \in \mathcal{K}}\left\|e^{i t \mathcal{L}} \varphi\right\|_{L_T^mW_x^{1,n}}<\frac{\rho}{2} .
 \end{equation}
for every $\rho>0$. By contradiction, there exists sequences $\left(\varphi_\tau\right) \subset \mathcal{K}$ and $\left(T_\tau\right) \subset \mathbb{R}^{+}$such that $T_\tau \rightarrow 0$ and
\begin{equation}\label{eq3.11}
  \left\|e^{i t \mathcal{L}} \varphi_\tau\right\|_{L_{T_\tau}^mW_x^{1,n}} \geq \bar{\rho}
\end{equation}
for a suitable $\bar{\rho}>0$. Since $\mathcal{K}$ is compact up to translation, passing to a subsequence, there exists a sequence $\left(x_\tau\right) \subset \mathbb{R}^d$ such that
$$
\widetilde{\varphi_\tau}(\cdot):=\varphi_n\left(\cdot-x_\tau\right) \xrightarrow{H^1(\mathbb{R}^d)} \varphi(\cdot)
$$
for a $\varphi \in H^1(\mathbb{R}^d)$. By continuity(induced by Strichartz estimates) we have, for every $ \overline{T}>0$,
$$
\left\|e^{i t \mathcal{L}} \widetilde{\varphi}_\tau\right\|_{L_{\overline{T}}^mW_x^{1,n}} \rightarrow\left\|e^{i t  \mathcal{L}} \varphi\right\|_{L_{\overline{T}}^mW_x^{1,n}} .
$$
Also, recording the translation invariance of Strichartz estimates we get from \eqref{eq3.11} that
$$
\left\|e^{i t \mathcal{L}} \widetilde{\varphi}_\tau\right\|_{L_{T_\tau}^mW_x^{1,n}}=\left\|e^{i t  \mathcal{L}} \varphi_\tau\right\|_{L_{T_\tau}^mW_x^{1,n}} \geq \bar{\rho}.
$$
Now, by Proposition \ref{P3.1} (iii), we have $e^{i t  \mathcal{L}} \varphi \in L_{1}^mW_x^{1,n}$, namely the function $t \rightarrow g(t):=\left\|e^{i t  \mathcal{L}} \varphi\right\|_{W^{1, n}(\mathbb{R}^d)}^m$ belongs to $L^1([0,1])$. Then by the Dominated Convergence Theorem we get $\left\|\chi_{[0, \tilde{T}]}(t) g(t)\right\|_{L^1([0,1])} \rightarrow 0$ as $\widetilde{T} \rightarrow 0$, namely $\left\|e^{i t  \mathcal{L}} \varphi\right\|_{L_{\widetilde{T} }^mW_x^{1,n}}^p \rightarrow 0$ as $\widetilde{T} \rightarrow 0$. Hence, we can choose $ \overline{T}>0$ such that
$$
\left\|e^{i t  \mathcal{L}} \varphi\right\|_{L_{\overline{T }}^mW_x^{1,n}}< \bar{\rho} .
$$
So we get a contradiction and the claim holds. Now, fix a $T>0$ such that \eqref{eq3.10} holds. By Proposition \ref{P3.1} (iii), it holds
$$
\left\|e^{i t  \mathcal{L}} \eta\right\|_{L_T^mW_x^{1,n}} \leq C\|\eta\|_{H^1(\mathbb{R}^d)}, \quad \forall \eta \in H^1(\mathbb{R}^d) .
$$
Thus, assuming that $\|\eta\|_{H^1(\mathbb{R}^d)}<\frac{\rho}{2 C}:=\varepsilon$, we obtain that
$$
\left\|e^{i t  \mathcal{L}} \eta\right\|_{L_T^mW_x^{1,n}}<\frac{\rho}{2}.
$$
Summarizing, we get that, for all $\varphi \in \mathcal{K}$ and all $\eta \in L_T^mW_x^{1,n}$ such that $\|\eta\|_{L_T^mW_x^{1,n}}<\varepsilon$,
$$
\left\|e^{i t  \mathcal{L}}(\varphi+\eta)\right\|_{L_T^mW_x^{1,n}} \leq\left\|e^{i t \mathcal{L}} \varphi\right\|_{L_T^mW_x^{1,n}}+\left\|e^{i t  \mathcal{L}} \eta\right\|_{L_T^mW_x^{1,n}}<\rho .
$$
This implies the Lemma.
\end{proof}
\begin{lemma}\label{L3.6}
Let $\mathcal{K} \subset H^1(\mathbb{R}^d) \backslash\{0\}$ be compact up to translation. Then, for every $\rho>0$ there exists $\varepsilon=\varepsilon(\rho)>0$ and $T=T(\rho)>0$ such that
$$
\sup\limits_{\varphi \in \mathcal{K}}\left\|e^{i t \mathcal{L}} \varphi\right\|_{\mathbf{X}_T}<\rho .
$$
\end{lemma}
\begin{proof}
We apply Lemma \ref{L3.5} twice with the admissible pairs $\left(m_1, n_1\right)$ and $\left(m_2, n_2\right)$. Then the proposition follows from the definition of the norm $\mathbf{X}_T$.
\end{proof}
\begin{lemma}\label{L3.7}
Let $\mathcal{K} \subset H^1(\mathbb{R}^d) \backslash\{0\}$ be compact up to translation. Then there exist $\varepsilon_0>0$ and $T_0>0$ such that the Cauchy problem \eqref{eq3.1}, where $\varphi$ satisfies $\mathop{\operatorname{dist}}\limits_{H^1(\mathbb{R}^d)}(\varphi, \mathcal{K})<\varepsilon_0$, has a unique solution on the time interval $\left[0, T_0\right]$ in the sense of Definition \ref{D3.2}.
\end{lemma}
\begin{proof}
We apply Lemma \ref{L3.6} where $\rho=\rho_0$ is given in Lemma \ref{L3.4}. Then Lemma \ref{L3.4} guarantees that the theorem holds for $\varepsilon_0=\varepsilon\left(\rho_0\right)>0$ and $T_0=\min \left\{T\left(\rho_0\right), 1\right\}>0$.
\end{proof}

\begin{lemma}\label{L3.8}
Let $v \in \mathcal{M}_a$. Then, for every $\varepsilon>0$ there exists $\delta>0$ such that
$$
\forall \varphi \in H^1(\mathbb{R}^d) \text { s.t. }\|\varphi-v\|_{H^1(\mathbb{R}^d)}<\delta \Longrightarrow \sup _{t \in\left[0, T_{\varphi}^{\text {max }}\right)} \mathop{\operatorname{dist}}\limits_{H^1(\mathbb{R}^d)}\left(u_{\varphi}(t), \mathcal{M}_a\right)<\varepsilon,
$$
where $u_{\varphi}(t)$ is the solution of \eqref{eq3.1}. In particular we have
\begin{equation*}
   u_{\varphi}(t)=m_a(t)+r(t),\ \forall t \in\left[0, T_{\varphi}^{\text {max }}\right), \|r(t)\|_{H^1(\mathbb{R}^d)}<\varepsilon,
\end{equation*}
where $m_a(t)\in\mathcal{M}_a$.
\end{lemma}
\begin{proof}
Suppose the theorem is false, there exists a decreasing sequence $\left(\delta_n\right) \subset \mathbb{R}^{+}$ converging to 0 and $\left(\varphi_n\right) \subset H^1(\mathbb{R}^d)$ satisfying
$$
\left\|\varphi_n-v\right\|_{H^1(\mathbb{R}^d)}<\delta_n,
$$
but
$$
\sup _{t \in\left[0, T_{\varphi}^{\text {max }}\right)} \mathop{\operatorname{dist}}_{H^1(\mathbb{R}^d)}\left(u_{\varphi_n}(t), \mathcal{M}_a\right)>\varepsilon_0
$$
for some $\varepsilon_0>0$. We observe that $\left\|\varphi_n\right\|_2^2 \rightarrow a$ and $I\left(\varphi_n\right) \rightarrow m(a)$ by continuity of $I(u)$. By conservation laws, for $n \in \mathbb{N}$ large enough, $u_{\varphi_n}$ will remains inside of $\mathbf{B}_{\rho_0}$ for all $t \in\left[0, T_{\varphi}^{\text {max }}\right)$. Indeed, if for some time $\bar{t}>0,\ \left\|\nabla u_{\varphi_n}(\bar{t})\right\|_2^2=\rho_0$, then in view of Lemma \ref{L1.4} we have that $I\left(u_{\varphi_n}(\bar{t})\right) \geq 0$ in contradiction with $m(a)<0$. Now let $t_n>0$ be the first time such that $\mathop{\operatorname{dist}}\limits_{H^1(\mathbb{R}^d)}\left(u_{\varphi_n}\left(t_n\right), \mathcal{M}_a\right)=\varepsilon_0$ and set $u_n:=u_{\varphi_n}\left(t_n\right)$. By conservation laws, $\left(u_n\right) \subset \mathbf{B}_{\rho_0}$ satisfies $\left\|u_n\right\|_2^2 \rightarrow a$ and $I\left(u_n\right) \rightarrow m(a)$ and thus, in view of Lemma \ref{L1.8}, it converges, up to translation, to an element of $\mathcal{M}_a$. Since $\mathcal{M}_a$ is invariant under translation this contradicts the equality $\mathop{\operatorname{dist}}\limits_{H^1(\mathbb{R}^d)}\left(u_n, \mathcal{M}_a\right)=\varepsilon_0>0$.
\end{proof}
\begin{lemma}\label{L3.9}
There exists a $\delta_0>0$ such that, if $\varphi \in H^1(\mathbb{R}^d)$ satisfies $\mathop{\operatorname{dist}}\limits_{H^1(\mathbb{R}^d)}\left(\varphi, \mathcal{M}_a\right)<\delta_0$ the corresponding solution to \eqref{eq3.1} satisfies $T_{\varphi}^{\max }=\infty$.
\end{lemma}
\begin{proof}
We make use of Lemma \ref{L3.7} where we choose $\mathcal{K}=\mathcal{M}_a$. We can choose a $\delta_0>0$ such that Lemma \ref{L3.8} holds for $\varepsilon=\varepsilon_0$ where $\varepsilon_0>0$ is given in Lemma \ref{L3.7}. Then Lemma \ref{L3.8} guarantees that the solution $u_{\varphi}(t)$ where $\mathop{\operatorname{dist}}\limits_{H^1(\mathbb{R}^d)}\left(u_{\varphi}(t), \mathcal{M}_a\right)<\delta_0$ satisfies $\mathop{\operatorname{dist}}\limits_{H^1(\mathbb{R}^d)}\left(u_{\varphi}(t), \mathcal{M}_a\right)<\varepsilon_0$ up to the maximum time of existence $T_{\varphi}^{\text {max }} \geq T_0$. Since, at any time in $\left(0, T_{\varphi}^{\max }\right)$ we can apply again Lemma \ref{L3.7} that guarantees an uniform additional time of existence $T_0>0$, this contradicts the definition of $T_{\varphi}^{\max }$ if $T_{\varphi}^{\max }<\infty$.
\end{proof}

\noindent\textbf{Proof of Theorem \ref{t1.2}.} The proof is a direct consequence of Lemma \ref{L3.8} and Lemma \ref{L3.9}.

\end{document}